\documentclass[a4paper,12pt,reqno]{amsart}
\usepackage{amsmath,amsthm,amssymb}
\usepackage{graphicx, color}
\usepackage[numbers,sort&compress]{natbib}
\nonstopmode \numberwithin{equation}{section}

\newtheorem{theorem}{Theorem}[section]

\newtheorem{remark}{Remark}[section]

\allowdisplaybreaks

\allowdisplaybreaks

\begin{document}

\title[A new extension of Hurwitz-Lerch Zeta function]
 {A new extension of Hurwitz-Lerch Zeta function}

\author[G. Rahman,   K. S. Nisar, M. Arshad ]{Gauhar Rahman,  Kottakkaran Sooppy Nisar*,  Muhammad Arshad}

\address{G. Rahman: Department of Mathematics, International Islamic University, Islamabad, Pakistan}
\email{gauhar55uom@gmail.com}

\address{ K. S. Nisar: Department of Mathematics, College of Arts and
Science at Wadi Aldawaser, 11991, Prince Sattam bin Abdulaziz University, Alkharj, Kingdom of Saudi Arabia}
\email{n.sooppy@psau.edu.sa, ksnisar1@gmail.com}
\address{M. Arshad: Department of Mathematics, International Islamic University, Islamabad, Pakistan}
\email{marshad$_-$zia@yahoo.com}

%\thanks{Submitted .}



\thanks{*Corresponding author}

\begin{abstract}
The main objective of this paper is to introduce a new extension of Hurwitz-Lerch Zeta function in terms of extended beta function. We then investigate its important properties such as integral representations, differential formulas,  Mellin transform and certain generating relations. Also, some important special cases of the main results are pointed out.
\end{abstract}

\maketitle

%%%%%%%%%%%%%%%%%%%%%%%%%%%%%%%%%%%%%%%%%%%%%%%%%%%%%%%%%%%%%%%%%%%%%%%%%%%%%%%%%%%%%%%%%%%%%%%%%%%%%%%%%%%%%%%%%%%%
\section{Introduction}\label{Intro1}
%%%%%%%%%%%%%%%%%%%%%%%%%%%%%%%%%%%%%%%%%%%%%%%%%%%%%%%%%%%%%%%%%%%%%%%%%%%%%%%%%%%%%%%%%%%%%%%%%%%%%%%%%%%%%%%%%%%%
The Hurwitz-Lerch Zeta function  and its integral representation are respectively defined by (see e.g., \cite{Erd1953} pp. 27, \cite{Sari2001}, pp. 121 and \cite{Sari2012}, pp. 194) as
\begin{eqnarray}\label{hur1}
\Phi(z,\sigma,a)=\sum_{n=0}^\infty\frac{z^n}{(n+a)^\sigma},
\end{eqnarray}
 $$(a\in\mathbb{C}\setminus \mathbb{Z}_0^-,\, \sigma\in\mathbb{C}\,\,\text{when} |z|<1\,; \Re(\sigma)>1\,\, \text{when}\, |z|=1).$$
  \begin{eqnarray}\label{hur2}
\Phi(z,\sigma,a)=\frac{1}{\Gamma(\sigma)}\int_0^\infty\frac{t^{\sigma-1}e^{-at}}{1-ze^{-t}}dt
=\frac{1}{\Gamma(\sigma)}\int_0^\infty\frac{t^{\sigma-1}e^{-(a-1)t}}{e^{t}-z}dt
\end{eqnarray}
$$(\Re(\sigma)>0\,,\Re(a)>0\,\, \text{when}\, |z|\leq1(z\neq1)\,; \Re(\sigma)>1\,\, \text{when}\, z=1).$$

Various generalization and extensions of the Hurwitz-Lerch Zeta function $\Phi(z,\sigma,a)$ have been studied by various researchers (see \cite{Chaudhry2001,Choi2008,Erd1953,Erd1955}. Goyal and Laddha \cite{Goyal} defined the following  Hurwitz-Lerch Zeta function and its integral representation by
\begin{eqnarray}\label{hur3}
\Phi_\upsilon^*(z,\sigma,a)=\sum_{n=0}^\infty\frac{(\upsilon)_n}{n!}\frac{z^n}{(n+a)^\sigma},
\end{eqnarray}
$$(\upsilon\in\mathbb{C}\,,a\in\mathbb{C}\setminus \mathbb{Z}_0^-,\,  \sigma\in\mathbb{C}\,\, \text{when}\, |z|<1\,, \Re(\sigma-\upsilon)>1\,\, \text{when}\, |z|=1).$$
  \begin{eqnarray}\label{hur4}
\Phi_\upsilon^*(z,\sigma,a)=\frac{1}{\Gamma(\sigma)}\int_0^\infty\frac{t^{\sigma-1}e^{-at}}{(1-ze^{-t})^{\upsilon}}dt
=\frac{1}{\Gamma(\sigma)}\int_0^\infty\frac{t^{\sigma-1}e^{-(a-\upsilon)t}}{(e^{t}-z)^\upsilon}dt
\end{eqnarray}
$$(\Re(\sigma)>0\,,\Re(a)>0\,\, \text{when}\, |z|\leq1\,(z\neq1)\,; \Re(\sigma)>1 \,\text{when}\, z=1).$$
Garg et al. \cite{Garg} defined an extension of (\ref{hur1}) and \eqref{hur3} and its integral representation as;
\begin{eqnarray}\label{hur5}
\Phi_{\gamma,\vartheta;\upsilon}(z,\sigma,a)=\sum_{n=0}^\infty\frac{(\gamma)_n(\vartheta)_n}{(\upsilon)_nn!}\frac{z^n}{(n+a)^\sigma},
\end{eqnarray}
$$(\gamma,\,\vartheta,\, \upsilon\in\mathbb{C}\,,a\in\mathbb{C}\setminus \mathbb{Z}_0^-,\, \sigma\in\mathbb{C}\,\, \text{when}\, |z|<1\,, \Re(\sigma+\upsilon-\gamma-\vartheta)>1\, \text{when}\, |z|=1).$$
\begin{eqnarray}\label{hur6}
\Phi_{\gamma,\vartheta;\upsilon}(z,\sigma,a)=\frac{1}{\Gamma(\sigma)}\int_0^\infty t^{\sigma-1}e^{-at}\,_2F_1\Big(\gamma,\vartheta;\upsilon;ze^{-t}\Big)dt
\end{eqnarray}
$$(\Re(\sigma)>0\,,\Re(a)>0\,\, \text{when}\, |z|\leq1\,(z\neq1)\,; \Re(\sigma)>1 \,\text{when}\, z=1).$$

An interesting extension of Hurwitz-Lerch Zeta function and its integral representation are found in \cite{Parmar} as
\begin{eqnarray}\label{hur7}
\Phi_{\gamma,\vartheta;\upsilon}(z,\sigma,a;p)=\sum_{n=0}^\infty\frac{(\gamma)_nB(\vartheta+n,\upsilon-\vartheta;p)}{B(\vartheta,\upsilon-\vartheta)n!}\frac{z^n}{(n+a)^\sigma},
\end{eqnarray}
$p\geq0\,,\gamma,\,\vartheta,\, \upsilon\in\mathbb{C}\,,a\in\mathbb{C}\setminus \mathbb{Z}_0^-,\, \sigma\in\mathbb{C}$ \,\text{when}\, $|z|<1$, $\Re(\sigma+\upsilon-\gamma-\vartheta)>1$\, \text{when}\, $|z|=1$ and $B(x,y;p)$ is the  extended beta function is defined by  \cite{Chaudhry1997} as
\begin{eqnarray}\label{Ebeta}
B(x,y;p)=B_p(x,y)=\int\limits_{0}^{1}t^{x-1}(1-t)^{y-1}
e^{-\frac{p}{t(1-t)}}dt
\end{eqnarray}
where $\Re(p)>0, \Re(x)>0, \Re(y)>0$ respectively. When $p=0$, then $B(x,y;0)=B(x,y)$.
\begin{eqnarray}\label{hur8}
\Phi_{\gamma,\vartheta;\upsilon}(z,\sigma,a)=\frac{1}{\Gamma(\sigma)}\int_0^\infty t^{\sigma-1}e^{-at}F_p\Big(\gamma,\vartheta;\upsilon;ze^{-t}\Big)dt
\end{eqnarray}
$(p\geq0\,, \Re(\sigma)>0\,,\Re(a)>0\,\, \text{when} |z|\leq1\,(z\neq1)\,; \Re(\sigma)>1$ when $z=1)$, and $F_p$ is the extended hypergeometric function (see \cite{Chaudhry2004}).

Very recently Shadab et al.  \cite{Choi2017} introduced a new and modified extension of beta function as:
\begin{eqnarray}\label{Sbeta}
B^{\lambda}_{p}(x,y)=B(x,y;p,\lambda)=\int_0^1t^{x-1}(1-t)^{y-1}E_{\lambda}\Big(-\frac{p}{t(1-t)}\Big)dt,
\end{eqnarray}
where $\Re(x)>0$, $\Re(y)>0$ and $E_{\lambda}\Big(.\Big)$ is Mittag-Leffler function defined by
\begin{eqnarray}
E_{\lambda}\Big(z\Big)=\sum_{n=0}^\infty\frac{z^n}{\Gamma(\lambda n+1)}.
\end{eqnarray}
Obviously, when $\lambda=1$  then $B^{1}_{p}(x,y)=B_p(x,y)$ is the extended beta function (see\cite{Chaudhry1997}). Similarly, when when $\lambda=1$ and $p=0$,  then $B^{1}_{0}(x,y)=B_0(x,y)$ is the classical beta function.

The extended hypergeometric function and its integral representation due to \cite{Choi2017} is defined by
\begin{align}\label{Shyper}
F_{p}^{\lambda}(\sigma_1,\sigma_2;\sigma_3;z)={}_2F_{1}\Big(\sigma_1,\sigma_2;\sigma_3;z;p,\lambda\Big)\notag&=
\sum_{n=0}^\infty(\sigma_1)_n\frac{B_{p}^{\lambda}(\sigma_2+n,\sigma_3-\sigma_2)}{B(\sigma_2,\sigma_3-\sigma_2)}
\frac{z^n}{n!}\notag\\
=&\sum_{n=0}^\infty(\sigma_1)_n\frac{B(\sigma_2+n,\sigma_3-\sigma_2; p,\lambda)}{B(\sigma_2,\sigma_3-\sigma_2)}
\frac{z^n}{n!},
\end{align}
where $p,\lambda\geq0$, $\sigma_1,\sigma_2,\sigma_3\in\mathbb{C}$  and $|z|<1$.

Many more generalizations and extensions on may be referred to the recent research papers (see \cite{Mubeen2017,Arshad1,Choi2014,Gauhar,Gauhar1,Mubeena2017,Ozarslan,Ozergin,Parmar2013})
%%%%%%%%%%%%%%%%%%%%%%%%%%%%%%%%%%%%%%%%%%%%%%%%%%%%%%%%%%%%%%%%%%%%%%%%%%%%%%%%%%%%%%%%%%%%%%%%
\section{Extension of Hurwitz-Lerch Zeta function and its properties}\label{sec-2}
%%%%%%%%%%%%%%%%%%%%%%%%%%%%%%%%%%%%%%%%%%%%%%%%%%%%%%%%%%%%%%%%%%%%%%%%%%%%%%%%%%%%%%%%%%%%%%%%
In this section, we introduce a new extension of Hurwitz-Lerch Zeta function in term of extended beta function (\ref{Sbeta}).
\begin{align}\label{hur9}
\Phi_{\gamma,\vartheta;\upsilon}(z,\sigma,a;p,\lambda)&=\Phi_{\gamma,\vartheta;\upsilon}^\lambda(z,\sigma,a;p)\notag\\
=&\sum_{n=0}^\infty\frac{(\gamma)_nB(\vartheta+n,\upsilon-\vartheta;p,\lambda)}{B(\vartheta,\upsilon-\vartheta)n!}\frac{z^n}{(n+a)^\sigma}\notag\\
=&\sum_{n=0}^\infty\frac{(\gamma)_nB_p^\lambda(\vartheta+n,\upsilon-\vartheta)}{B(\vartheta,\upsilon-\vartheta)n!}\frac{z^n}{(n+a)^\sigma}
\end{align}
$(p\geq0\,,\lambda>0\,,\gamma,\,\vartheta,\, \upsilon\in\mathbb{C}\,,a\in\mathbb{C}\setminus \mathbb{Z}_0^-,\, \sigma\in\mathbb{C}$ when $|z|<1$, $\Re(\sigma+\upsilon-\gamma-\vartheta)>1$ when $|z|=1)$.
\begin{remark}
 The following special cases can be derived from extension of Hurwitz-Lerch Zeta function $\Phi_{\gamma,\vartheta;\upsilon}(z,\sigma,a;p,\lambda)$ defined in \eqref{hur9}.
\end{remark}

\textbf{Case 1.} If we consider $\gamma=1$, then \eqref{hur9} will lead to another extension of generalized Hurwitz-Lerch Zeta function $\Phi_{\vartheta;\upsilon}^{1,1}(z,\sigma,a;p,\lambda)$ introduced earlier by Lin and Srivastava \cite{Lin} with $\rho=\sigma=1$.

\begin{align}\label{hur10}
\Phi_{\vartheta;\upsilon}^{1,1}(z,\sigma,a;p,\lambda)&=\Phi_{1,\vartheta;\upsilon}(z,\sigma,a;p,\lambda)\notag\\
=&\sum_{n=0}^\infty\frac{B(\vartheta+n,\upsilon-\vartheta;p,\lambda)}{B(\vartheta,\upsilon-\vartheta)n!}\frac{z^n}{(n+a)^\sigma}
\end{align}
$(p\geq0\,,\lambda>0\,,\vartheta,\, \upsilon\in\mathbb{C}\,,a\in\mathbb{C}\setminus \mathbb{Z}_0^-,\, \sigma\in\mathbb{C}$ when $|z|<1$, $\Re(\sigma+\upsilon-\gamma-\vartheta)>1$ when $|z|=1)$.

\textbf{Case 2.} If we set $\gamma=\upsilon=1$ in \eqref{hur9}, then we get a new extension of  Hurwitz-Lerch Zeta function $\Phi_{\vartheta}^{*}(z,\sigma,a;p,\lambda)$ which is the extension of the function defined by Goyal and Laddha \cite{Goyal} as:
\begin{align}\label{hur11}
\Phi_{\vartheta}^{*}(z,\sigma,a;p,\lambda)&=\Phi_{1,\vartheta;1}(z,\sigma,a;p,\lambda)\notag\\
=&\sum_{n=0}^\infty\frac{B(\vartheta+n,1-\vartheta;p,\lambda)}{B(\vartheta,1-\vartheta)n!}\frac{z^n}{(n+a)^\sigma}
\end{align}
$(p\geq0\,,\lambda>0\,,\vartheta\in\mathbb{C}\,,a\in\mathbb{C}\setminus \mathbb{Z}_0^-,\, \sigma\in\mathbb{C}$ when $|z|<1$, $\Re(\sigma+\upsilon-\gamma-\vartheta)>1$ when $|z|=1)$.

\textbf{Case 3.} A limit case of  extended    Hurwitz-Lerch Zeta function $\Phi_{\vartheta}^{*}(z,\sigma,a;p,\lambda)$  defined by \eqref{hur9} is given by
\begin{align}\label{hur12}
\Phi_{\vartheta;\upsilon}^{*}(z,\sigma,a;p,\lambda)&=\lim_{|\gamma|\rightarrow\infty}\Big\{\Phi_{\gamma,\vartheta;\upsilon}\Big(\frac{z}{\gamma},\sigma,a;p,\lambda\Big)\Big\}\notag\\
=&\sum_{n=0}^\infty\frac{B(\vartheta+n,\upsilon-\vartheta;p,\lambda)}{B(\vartheta,\upsilon-\vartheta)n!}\frac{z^n}{(n+a)^\sigma}
\end{align}
$(p\geq0\,,\lambda>0\,,\vartheta\in\mathbb{C}\,,a\in\mathbb{C}\setminus \mathbb{Z}_0^-,\, \sigma\in\mathbb{C}$ when $|z|<1$, $\Re(\sigma+\upsilon-\gamma-\vartheta)>1$ when $|z|=1)$.
\begin{remark}
It is clear that,\\
(i) if $\lambda=1$ then \eqref{hur9} reduces to  \eqref{hur7}.\\
(ii) if $\lambda=1$ and $p=0$ then \eqref{hur9} gives \eqref{hur5}.\\
(iii) if $\lambda=\gamma=\upsilon=1$ and $p=0$ then \eqref{hur9} turns to \eqref{hur3}.\\
(iv) if $\lambda=\gamma=\vartheta=\upsilon=1$ and $p=0$ then \eqref{hur9} reduces to \eqref{hur1}.\\
\end{remark}
%%%%%%%%%%%%%%%%%%%%%%%%%%%%%%%%%%%%%%%%%%%%%%%%%%%%%%%%%%%%%%%%%%%%%%%%%%%%%%%%%%%%%%%%%%%%%%%%%%%%%%%%%%%%%
\section{Integral representations and derivative formulas}\label{sec-3}
%%%%%%%%%%%%%%%%%%%%%%%%%%%%%%%%%%%%%%%%%%%%%%%%%%%%%%%%%%%%%%%%%%%%%%%%%%%%%%%%%%%%%%%%%%%%%%%%%%%%%%%%%%%%%%
In this section, we define varieties of integrals and a derivative formula for \eqref{hur9}.
\begin{theorem}\label{th1}
The following integral representation holds true:
\begin{eqnarray}\label{int1}
\Phi_{\gamma,\vartheta;\upsilon}(z,\sigma,a;p,\lambda)=\frac{1}{\Gamma(\sigma)}\int_0^\infty t^{\sigma-1}e^{-at}F_p^\lambda\Big(\gamma,\vartheta;\upsilon;ze^{-t}\Big)dt
\end{eqnarray}
$(p\geq0\,,\lambda>0\,,\vartheta\in\mathbb{C}\,,a\in\mathbb{C}\setminus \mathbb{Z}_0^-,\, \sigma\in\mathbb{C}$ \emph{when} $|z|<1$, $\Re(\sigma+\upsilon-\gamma-\vartheta)>1$ \emph{when} $|z|=1)$.
\end{theorem}
\begin{proof}
Using the following Euler integral formula:
\begin{eqnarray}\label{int2}
\frac{1}{(n+a)^\sigma}=\frac{1}{\Gamma(\sigma)}\int_0^\infty t^{\sigma-1} e^{-(n+a)t}dt,
\end{eqnarray}
where $\min\{\Re(\sigma), \Re(a)\}>0; n\in\mathbb{N}_0$ in \eqref{hur9}, we have
\begin{align}\label{int3}
\Phi_{\gamma,\vartheta;\upsilon}(z,\sigma,a;p,\lambda)
=&\sum_{n=0}^\infty\frac{(\gamma)_nB(\vartheta+n,\upsilon-\vartheta;p,\lambda)}{B(\vartheta,\upsilon-\vartheta)n!}
\Big(\frac{1}{\Gamma(\sigma)}\int_0^\infty t^{\sigma-1} e^{-(n+a)t}dt\Big).
\end{align}
Interchanging the order of summation and integration under the condition stated in Theorem \ref{th1}, we get
\begin{align*}
\Phi_{\gamma,\vartheta;\upsilon}(z,\sigma,a;p,\lambda)
=&\frac{1}{\Gamma(\sigma)}\int_0^\infty t^{\sigma-1} e^{-(n+a)t}
\Big(\sum_{n=0}^\infty\frac{(\gamma)_nB(\vartheta+n,\upsilon-\vartheta;p,\lambda)}{B(\vartheta,\upsilon-\vartheta)n!}\Big)dt,
\end{align*}
by using \eqref{Shyper} we arrived the desired result.
\end{proof}
%%%%%%%%%%%%%%%%%%%%%%%%%%%%%%%%%%%%%%%%%%%%%%%%%%%%%%%%%%%%%%%%%%%%%%%%%%%%%%%%%%%%%%%
\begin{theorem}\label{th2}
The following integral representation holds true:
\begin{align}\label{int4}
\Phi_{\gamma,\vartheta;\upsilon}(z,\sigma,a;p,\lambda)&=\frac{1}{B(\vartheta,\upsilon-\vartheta)}\int_0^\infty\frac{x^{\vartheta-1}}{(1+x)^\upsilon}\notag\\
&\times E_\lambda\Big(-p(2+x+\frac{1}{x})\Big)
\Phi_\gamma^*\Big(\frac{zx}{1+x},\sigma,a\Big)dx
\end{align}
$$(p\geq0,\, \lambda>0\, , \Re(\upsilon)>\Re(\vartheta)>0)$$
and
\begin{align}\label{int5}
\Phi_{\gamma,\vartheta;\upsilon}(z,\sigma,a;p,\lambda)&=\frac{1}{B(\vartheta,\upsilon-\vartheta)}\int_0^\infty\frac{t^{\sigma-1}e^{-at}x^{\vartheta-1}}{(1+x)^\upsilon}\notag\\
&\times E_\lambda\Big(-p(2+x+\frac{1}{x})\Big)
\Big(1-\frac{zxe^{-t}}{1+x}\Big)^{-\gamma}dx
\end{align}
$$(p\geq0,\, \lambda>0\, , \Re(\upsilon)>\Re(\vartheta)>0\,, \min\{\Re(\sigma),\Re(a)\}>0),$$
provided that both the integrals converges.
\end{theorem}
\begin{proof}
Consider the following integral representation of extended beta function (see \cite{Choi2017})
\begin{eqnarray}\label{int6}
B(\sigma_1,\sigma_2;p,\lambda)=\int_0^\infty\frac{x^{\sigma_1-1}}{(1+x)^{\sigma_1+\sigma_2}}
E_\lambda\Big(-p(2+x+\frac{1}{x})\Big)dx, \, \Re(p)>0.
\end{eqnarray}
Substituting $\sigma_1=\vartheta+n$ and $\sigma_2=\upsilon-\vartheta$ in \eqref{int6}, we have
\begin{eqnarray}\label{int7}
B(\vartheta+n,\upsilon-\vartheta;p,\lambda)=\int_0^\infty\frac{x^{\vartheta+n-1}}{(1+x)^{\upsilon+n}}
E_\lambda\Big(-p(2+x+\frac{1}{x})\Big)dx, \, \Re(p)>0.
\end{eqnarray}
Using \eqref{int7} in \eqref{hur9}, we have
\begin{align}\label{int8}
\Phi_{\gamma,\vartheta;\upsilon}(z,\sigma,a;p,\lambda)&=
\frac{1}{B(\vartheta,\upsilon-\vartheta)}\int_0^\infty\frac{x^{\vartheta-1}}{(1+x)^{\upsilon}}\notag\\
&\times E_\lambda\Big(-p(2+x+\frac{1}{x})\Big)\Big(\sum_{n=0}^\infty\frac{(\gamma)_n}{n!}\frac{(xz)^n}{(1+x)^{n}(n+a)^\sigma}\Big)dx.
\end{align}
In view of \eqref{hur3}, we get the required result \eqref{int4}.

Now, from \eqref{int1} and \eqref{int8}, we have
\begin{align}
&\Phi_{\gamma,\vartheta;\upsilon}(z,\sigma,a;p,\lambda)\notag\\
=&\frac{1}{\Gamma(\sigma)}\int_0^\infty t^{\sigma-1} e^{-(n+a)t}
\Big(\sum_{n=0}^\infty\frac{(\gamma)_nB(\vartheta+n,\upsilon-\vartheta;p,\lambda)}{B(\vartheta,\upsilon-\vartheta)n!}\Big)dt\\
=&\frac{1}{\Gamma(\sigma)B(\vartheta,\upsilon-\vartheta)}\int_0^\infty\int_0^\infty\frac{x^{\vartheta-1}t^{\sigma-1} e^{-at}}{(1+x)^{\upsilon}}\notag\\
&\times E_\lambda\Big(-p(2+x+\frac{1}{x})\Big)\Big(\sum_{n=0}^\infty\frac{(\gamma)_n}{n!}\frac{(xze^{-t})^n}{(1+x)^{n}(n+a)^\sigma}\Big)dtdx.
\end{align}
By using the binomial expansion
$$(1-zt)^{-\alpha}=\sum_{n=0}^\infty\frac{(\alpha)_n(zt)^n}{n!},$$
we get the desired result.
\end{proof}
%%%%%%%%%%%%%%%%%%%%%%%%%%%%%%%%%%%%%%%%%%%%%%%%%%%%%%%%%%%%%%%%%%%%%%%%%%%%%%%%%%%%%%%%%%%%
\begin{theorem}\label{th3}
The following integral representation holds true:
\begin{align}
\Phi_{\gamma,\vartheta;\upsilon}(z,\sigma,a;p,\lambda)=\frac{1}{\Gamma(\gamma)}\int_0^\infty t^{\gamma-1}e^{-t}\Phi_{\vartheta;\upsilon}^*(zt,\sigma,a;p,\lambda)dt
\end{align}
$(p\geq0,\,\lambda>0,\, \Re(\gamma)>0,\, \Re(a)>0$, \emph{when} $|z|\leq1 (z\neq1); \Re(\sigma)>1$, \emph{when} $z=1$), where $\Phi_{\vartheta;\upsilon}^*(zt,\sigma,a;p,\lambda)$ is the limiting case defined by \eqref{hur12}.
\end{theorem}
\begin{proof}
The integral representation of Pochhammer symbol $(\gamma)_n$ given in \cite{Sari2012} as
\begin{eqnarray}\label{int9}
(\gamma)_n=\frac{1}{\Gamma(\gamma)}\int_0^\infty t^{\gamma+n-1}e^{-t}dt.
\end{eqnarray}
The use of \eqref{int9} in \eqref{hur9} and interchanging the order of summation and integration under the condition of Theorem \ref{th3}, we have
 \begin{eqnarray}\label{int10}
\Phi_{\gamma,\vartheta;\upsilon}(z,\sigma,a;p,\lambda)=\frac{1}{\Gamma(\gamma)}\int_0^\infty t^{\gamma-1}e^{-t}\sum_{n=0}^\infty\frac{B(\vartheta+n,\upsilon-\vartheta;p,\lambda)}{B(\vartheta,\upsilon-\vartheta)}\frac{(zt)^n}{n!(n+a)^\sigma}dt.
\end{eqnarray}
Now, by using \eqref{hur12}, we get the desired proof.
\end{proof}
%%%%%%%%%%%%%%%%%%%%%%%%%%%%%%%%%%%%%%%%%%%%%%%%%%%%%%%%%%%%%%%%%%%%%%%%%%%%%%%
Next, we derive a derivative formula of $\Phi_{\gamma,\vartheta;\upsilon}(z,\sigma,a;p,\lambda)$.
\begin{theorem}
The following formula holds true
\begin{align}\label{der}
\frac{d^n}{dz^n}\{\Phi_{\gamma,\vartheta;\upsilon}(z,\sigma,a;p,\lambda)\}=\frac{(\gamma)_n(\vartheta)_n}{(\upsilon)_n}
\Phi_{\gamma+n,\vartheta+n;\upsilon+n}(z,\sigma,a+n;p,\lambda),\, (n\in\mathbb{N}_0).
\end{align}
\end{theorem}
\begin{proof}
 Differentiation of \eqref{hur9} with respect to $z$ yields
\begin{align}\label{der1}
\frac{d}{dz}\{\Phi_{\gamma,\vartheta;\upsilon}(z,\sigma,a;p,\lambda)\}=
\sum_{n=1}^\infty\frac{(\gamma)_n}{(n-1)!}\frac{B(\vartheta+n,\upsilon-\vartheta;p,\lambda)}{B(\vartheta,\upsilon-\vartheta)}\frac{z^{n-1}}{(n+a)^\sigma}.
\end{align}
Replacing $n$ by $n+1$  and applying the following well known identities
$$ B(b,c-b)=\frac{c}{b}B(b+1,c-b)\,\, and\,\, (a)_{n+1}=a (a+1)_n$$
in the right hand side of \eqref{der1}, we have
\begin{align}\label{der2}
\frac{d}{dz}\{\Phi_{\gamma,\vartheta;\upsilon}(z,\sigma,a;p,\lambda)\}=\frac{\gamma\vartheta}{\upsilon}
\Phi_{\gamma+1,\vartheta+1;\upsilon+1}(z,\sigma,a+1;p,\lambda).
\end{align}
Again, differentiating \eqref{der2} with respect to $z$, we obtain
\begin{align}\label{der3}
\frac{d^2}{dz^2}\{\Phi_{\gamma,\vartheta;\upsilon}(z,\sigma,a;p,\lambda)\}=\frac{\gamma(\gamma+1)\vartheta(\vartheta+1)}{\upsilon(\upsilon+1)}
\Phi_{\gamma+2,\vartheta+2;\upsilon+2}(z,\sigma,a+2;p,\lambda).
\end{align}
Continuing up to $n$-times gives the required proof.
\end{proof}
%%%%%%%%%%%%%%%%%%%%%%%%%%%%%%%%%%%%%%%%%%%%%%%%%%%%%%%%%%%%%%%
\section{Mellin Transformation and generating relations}\label{sec-4}
%%%%%%%%%%%%%%%%%%%%%%%%%%%%%%%%%%%%%%%%%%%%%%%%%%%%%%%%%%%%%%
In this section, we derive Mellin transformation and some generating relations for \eqref{hur9}.
Here, we recall the following well-known Mellin transform of integrable function as:
\begin{eqnarray}\label{Mel}
\mathfrak{M}\{f(t); t\rightarrow r\}=\int_0^\infty t^{r-1}f(t)dt.
\end{eqnarray}
\begin{theorem}\label{th5}
The Mellin transform of the extended Hurwitz-Lerch Zeta function is given by
\begin{align}\label{Mel2}
\mathfrak{M}\{\Phi_{\gamma,\vartheta;\upsilon}(z,\sigma,a;p,\lambda); p\rightarrow r\}=
\frac{\pi}{\sin(r\pi)}\frac{B(\vartheta+r,\upsilon-\vartheta+r)}{\Gamma(1-r\lambda)B(\vartheta,\upsilon-\vartheta)}\Phi_{\gamma,\vartheta+r;\upsilon+2r}(z,\sigma,a)
\end{align}
$$(\Re(r)>0\,\, \emph{and}\,\, \Re(\vartheta+r)>0).$$
\end{theorem}
\begin{proof}
The Mellin Transformation of \eqref{hur9} is given by:
\begin{align}\label{Mel3}
\mathfrak{M}\{\Phi_{\gamma,\vartheta;\upsilon}(z,\sigma,a;p,\lambda); p\rightarrow r\}&=\int_0^\infty p^{r-1}\Phi_{\gamma,\vartheta;\upsilon}(z,\sigma,a;p,\lambda)dp\notag\\
&=\int_0^\infty p^{r-1}\Big(\sum_{n=0}^\infty\frac{(\gamma)_n}{n!}\frac{B(\vartheta+n,\upsilon-\vartheta;p,\lambda)}{B(\vartheta,\upsilon-\vartheta)}\frac{z^n}{(n+a)^\sigma}\Big)dp\notag\\
&=\frac{1}{B(\vartheta,\upsilon-\vartheta)}\sum_{n=0}^\infty\frac{(\gamma)_n}{n!}
\frac{z^n}{(n+a)^\sigma}\notag\\
&\times\int_0^\infty p^{r-1}B(\vartheta+n,\upsilon-\vartheta;p,\lambda)dp\notag\\
&=\frac{1}{B(\vartheta,\upsilon-\vartheta)}\sum_{n=0}^\infty\frac{(\gamma)_n}{n!}
\frac{(zt)^n}{(n+a)^\sigma}\int_0^\infty p^{r-1}\notag\\
&\times\int_0^1t^{\vartheta-1}(1-t)^{\upsilon-\vartheta-1}E_\lambda\Big(-\frac{p}{t(1-t)}\Big)dtdp.
\end{align}
Interchanging the order of integrations in \eqref{Mel3}, we have
\begin{align}\label{Mel4}
\mathfrak{M}\{\Phi_{\gamma,\vartheta;\upsilon}(z,\sigma,a;p,\lambda); p\rightarrow r\}&=\frac{1}{B(\vartheta,\upsilon-\vartheta)}\sum_{n=0}^\infty\frac{(\gamma)_n}{n!}
\frac{(zt)^n}{(n+a)^\sigma}\int_0^1t^{\vartheta-1}(1-t)^{\upsilon-\vartheta-1}\notag\\
&\times\Big(\int_0^\infty p^{r-1}E_\lambda\Big(-\frac{p}{t(1-t)}\Big)dp\Big)dt.
\end{align}
 The second integrand in \eqref{Mel4} becomes
\begin{eqnarray}\label{Mel5}
\int_0^\infty p^{r-1}E_\lambda\Big(-\frac{p}{t(1-t)}\Big)=t^r(1-t)^r\int_0^\infty v^{r-1}E_\lambda\Big(-v\Big)dv.
\end{eqnarray}
and recall
\begin{align}\label{gamma}
\int_{0}^{\infty}v^{r-1}E_{\lambda,\gamma}^{\delta}(-wv)dv=\frac{\Gamma(r)\Gamma(\delta-r)}{\Gamma(\delta)w^{r}
\Gamma(\gamma-r\lambda)}.
\end{align}
For $\gamma=\delta=1$ and $w=1$, \eqref{gamma} becomes
\begin{align}\label{gamma1}
\int_{0}^{\infty}v^{r-1}E_{\lambda}(-v)dv=\frac{\Gamma(r)\Gamma(1-r)}{\Gamma(1-r\lambda)}.
\end{align}
Using (\ref{gamma1}) in (\ref{Mel5}), we have
\begin{eqnarray}\label{Mel6}
\int_0^\infty p^{r-1}E_\lambda\Big(-\frac{p}{t(1-t)}\Big)=t^r(1-t)^r\frac{\Gamma(r)\Gamma(1-r)}{\Gamma(1-r\lambda)}.
\end{eqnarray}
Thus by using \eqref{Mel6} and the following Euler's reflection formula of Gamma function
\begin{align}
{\Gamma(r)\Gamma(1-r)}=\frac{\pi}{\sin(\pi r)},
\end{align}
in \eqref{Mel4}, we get the desired result.
\end{proof}
\begin{remark}

The  special case of \eqref{hur9} for $r=1$ gives the following integral representation
\begin{align}
\int_0^\infty \Phi_{\gamma,\vartheta;\upsilon}(z,\sigma,a;p,\lambda)dp=
\frac{\pi}{\sin(\pi)}\frac{B(\vartheta+1,\upsilon-\vartheta+1)}{\Gamma(1-\lambda)B(\vartheta,\upsilon-\vartheta)}\Phi_{\gamma,\vartheta+1;\upsilon+2}(z,\sigma,a)
\end{align}
\end{remark}
%%%%%%%%%%%%%%%%%%%%%%%%%%%%%%%%%%%%%%%%%%%%%%%%%%%%%%%%%%%%%%%%%%%%%%%%%%%%%%
Next, we derive the generating function for the extended Hurwitz-Lerch Zeta function defined by \eqref{hur9}.
\begin{theorem}\label{th6}
The following generating function for $\Phi_{\gamma,\vartheta;\upsilon}(z,\sigma,a;p,\lambda)$ holds true:
\begin{align}\label{g}
\sum_{n=0}^\infty(\gamma)_n \Phi_{\gamma+n,\vartheta;\upsilon}(z,\sigma,a;p,\lambda)\frac{t^n}{n!}=(1-t)^{-\gamma}
\Phi_{\gamma+n,\vartheta;\upsilon}(\frac{z}{1-t},\sigma,a;p,\lambda)
\end{align}
$$(p\geq0\,, \lambda>0, \gamma\, , \vartheta\,, \upsilon\in\mathbb{C}\,\, \emph{and}\,\, |t|<1).$$
\end{theorem}
\begin{proof}
Let the left hand side of \eqref{g} be denoted by $\mathfrak{L}$, the from \eqref{hur9}, we have

 \begin{align}\label{g1}
\mathfrak{L}=\sum_{n=0}^\infty(\gamma)_n\Big\{\sum_{k=0}^\infty(\gamma+n)_k\frac{B(\vartheta+k,\upsilon-\vartheta;p,\lambda)}{B(\vartheta,\upsilon-\vartheta)}
\frac{z^k}{k!(k+a)^\sigma}\Big\}\frac{t^n}{n!}
\end{align}
Interchanging the order of summations and using the identity $(\gamma)_n(\gamma+n)_k=(\gamma)_k(\gamma+k)_n$ in \eqref{g1}, we have
 \begin{align}\label{g2}
\mathfrak{L}=\sum_{k=0}^\infty(\gamma)_k\frac{B(\vartheta+k,\upsilon-\vartheta;p,\lambda)}{B(\vartheta,\upsilon-\vartheta)}
\Big\{\sum_{n=0}^\infty(\gamma+k)_n\frac{t^n}{n!}\Big\}
\frac{z^k}{k!(k+a)^\sigma}
\end{align}
Now, using the following binomial expansion
\begin{eqnarray*}
(1-t)^{-\gamma-k}=\sum_{n=0}^\infty(\gamma+k)_n\frac{t^n}{n!}\,\, (|t|<1),
\end{eqnarray*}
and interpreting in term of \eqref{hur9} as a function of the form $\Phi_{\gamma+n,\vartheta;\upsilon}(\frac{z}{1-t},\sigma,a;p,\lambda)$, which completes the proof of Theorem \eqref{th6}.
\end{proof}
%%%%%%%%%%%%%%%%%%%%%%%%%%%%%%%%%%%%%%%%%%%%%%%%%%%%%%%%%%%%%%%%%%%%%%%%%%%%%%%%%%%%%%%%%%%%%
\begin{theorem}
The following generating function for $\Phi_{\gamma,\vartheta;\upsilon}(z,\sigma,a;p,\lambda)$ holds true:
\begin{align}\label{g3}
\sum_{n=0}^\infty\frac{(\sigma)_n}{n!} \Phi_{\gamma,\vartheta;\upsilon}(z,\sigma+n,a;p,\lambda)t^n=
\Phi_{\gamma,\vartheta;\upsilon}(z,\sigma,a-t;p,\lambda)
\end{align}
$$(p\geq0\,, \lambda>0, \gamma\, , \vartheta\,, \upsilon\in\mathbb{C}\,\,\, \emph{and}\,\,\, |t|<a;\,\,\sigma\neq1).$$
\end{theorem}
\begin{proof}
Applying \eqref{hur9} to the right hand side of \eqref{g3}, we have
\begin{align}\label{g4}
\Phi_{\gamma,\vartheta;\upsilon}(z,\sigma,a-t;p,\lambda)&=\sum_{k=0}^\infty(\gamma)_k\frac{B(\vartheta+k,\upsilon-\vartheta;p,\lambda)}{B(\vartheta,\upsilon-\vartheta)}
\frac{z^k}{k!(k+a-t)^\sigma}\notag\\
=&\sum_{k=0}^\infty(\gamma)_k\frac{B(\vartheta+k,\upsilon-\vartheta;p,\lambda)}{B(\vartheta,\upsilon-\vartheta)}
\frac{z^k}{k!(k+a)^\sigma}\Big(1-\frac{t}{k+a}\Big)^{-\sigma}.
\end{align}
 Using the following binomial expansion
\begin{eqnarray*}
(1-t)^{-\gamma-k}=\sum_{n=0}^\infty(\gamma+k)_n\frac{t^n}{n!}\,\, (|t|<1),
\end{eqnarray*}
in \eqref{g4}, we have
\begin{align}
\Phi_{\gamma,\vartheta;\upsilon}(z,\sigma,a-t;p,\lambda)&
=\sum_{k=0}^\infty(\gamma)_k\frac{B(\vartheta+k,\upsilon-\vartheta;p,\lambda)}{B(\vartheta,\upsilon-\vartheta)}
\frac{z^k}{k!(k+a)^\sigma}\Big\{\sum_{n=0}^\infty\frac{(\sigma)_n}{n!}\frac{t^n}{(k+\sigma)^n}\notag\\
=&\sum_{n=0}^\infty\frac{(\sigma)_n}{n!}\Big\{\sum_{k=0}^\infty(\gamma)_k\frac{B(\vartheta+k,\upsilon-\vartheta;p,\lambda)}{B(\vartheta,\upsilon-\vartheta)}
\frac{z^k}{k!(k+a)^{\sigma+n}}\Big\}t^n
\end{align}
By making the use of \eqref{hur9}, we get the desired result.
\end{proof}
%%%%%%%%%%%%%%%%%%%%%%%%%%%%%%%%%%%%%%%%%%%%%%%%%%%%%%%%%%%%%%%%%%%%%%%%%%%%%%%%%%%%%%%%%%%%%%%%%
\section{Concluding Remark}
In this paper, we established a new extension of extended Hurwitz-Lerch Zeta functions which are earlier studied by many researchers cited therein literature. If we consider $\lambda=1$, then the results established here will reduce to the results studied by Parmar et al. \cite{Parmar}. Silmilarly, if we consider $\lambda=1$ and $p=0$ then the result will reduce to the work of Garg et al. \cite{Garg}.

%%%%%%%%%%%%%%%%%%%%%%%%%%%%%%%%%%%%%%%%%%%%%%%%%%%%%%%%%%%%%%%%%%%%%%%%%%%%%%%%%%%%%%%%%%%%%%%%%%%%%%%%%%%%%%%%%%%%%%
{\bf Conflict of Interests}\\
The author(s) declare(s) that there is no conflict of interests regarding the
publication of this article.

 %%%%%%%%%%%%%%%%%%%%%%%%%%%%%%%%%%%%%%%%%%%%%%%%%%%%%%%%%%%%%%%%%%%%%%%%%%%%%%%%%%%%%%%%%%%%%%%%%%%%%%%%%%%%%%%
 
\end{document}